\documentclass[12pt,leqno]{amsart}

\usepackage{amsfonts,amssymb}
\usepackage{graphicx,tikz}
\usepackage{tikz-cd} 
\usepackage{float}
\usepackage{amsmath, amsthm, amsfonts}
\usepackage{hyperref}
\usepackage{enumitem}  
\usepackage[capitalise]{cleveref}
\usepackage{comment}
\usepackage{enumitem}
\usepackage{amsrefs}
\usepackage{nicefrac}
\usepackage[official]{eurosym}
\usepackage{graphicx}
\usepackage{nomencl}
\usepackage{amsfonts}
\usepackage{hyperref}
\usepackage{tabto}
\usepackage{amssymb}
\usepackage{fancyhdr}
\usepackage{amscd}
\usepackage[english]{babel}

\usepackage[margin=1.25in]{geometry}

\newtheorem{theorem}{\sc Theorem}[section]
\newtheorem{lemma}[theorem]{\sc Lemma}

\author[P. Shumyatsky]{Pavel Shumyatsky} \address{Pavel Shumyatsky:  Department of Mathematics, University of Brasilia, 70910-900 Brasilia, Brazil} \email{pavel@unb.br}

\author[A. Thillaisundaram]{Anitha Thillaisundaram} 
\address{Anitha Thillaisundaram: Centre for Mathematical Sciences, Lund University,  223 62 Lund, Sweden}
\email{anitha.thillaisundaram@math.lu.se}

\date{\today}

\thanks{This research was supported by the  Royal Physiographic Society of Lund. The first author was also supported by FAPDF}


	\title{Profinite groups with soluble centralisers}

\begin{document}

\begin{abstract}
We show that a profinite group, in which the centralisers of non-trivial elements are metabelian, is either virtually pro-$p$ or virtually soluble of derived length at most 4. We furthermore show that a prosoluble group, in which the centralisers of non-trivial elements are soluble of bounded derived length, is  either soluble or virtually pro-$p$. 
\end{abstract}
\maketitle

\section{Introduction}

Groups where the centralisers of non-trivial elements satisfy certain properties is a well-established subject. Consider for instance groups with abelian centralisers, also called CA-groups. Finite CA-groups were classified by several mathematicians throughout the 20th century, including Suzuki~\cite{Suzuki}, who proved that simple CA-groups have even order. Suzuki’s work was an important precursor of the remarkable Feit–Thompson theorem~\cite{FT} on the solubility of groups of odd order. Centralisers of non-trivial elements also play an important role in coprime group actions, hence imposing conditions on centralisers in a group acting coprimely on another group has received more consideration; see~\cite{Khukhro} and references therein. 

Imposing restrictions  on the properties of the centralisers naturally has an effect on the structure of the corresponding  group. The study of this effect on a profinite group  is relatively new,  see~\cite{zapa, S1, SZ, ccm, AS}, with the pioneering work in this direction being the study of profinite CA-groups  in~\cite{zapa}, where it was shown that a profinite CA-group is, up to finite index, an abelian group or a pro-$p$ group. This result was then generalised in~\cite{S1} to profinite groups with pronilpotent centralisers, and  all such groups were classified up to finite index. We also mention an earlier work by Shalev~\cite{Shalev}, who considered the slightly different situation of a profinite group in which the centraliser of any element is either finite or of finite index, and he showed that such a group is virtually abelian. Some recent work in this direction includes \cite{DMS} and \cite{AS}.

In this paper we are interested in studying profinite groups in which the centraliser of any non-trivial element is  soluble of bounded derived length $d$, with special emphasis on the case $d=2$. For a positive integer~$d$, following the notation in~\cite{casolo}, we denote by $\text{CA}_d$ the class of groups in which the centraliser of any non-trivial element is soluble of  derived length at most~$d$.  Finite soluble $\text{CA}_d$-groups were studied by Casolo and Jabara in~\cite{casolo}, where they gave bounds on the derived length of the group in terms of~$d$; see Lemma~\ref{lem:casolo}.

The following is our main result.
\begin{theorem}\label{thm:main}
    Let $G$ be a profinite $\textup{CA}_2$-group.     Then $G$  is either virtually pro-$p$ or virtually soluble of derived length at most 4.
\end{theorem}

The bound 4 for the above derived length is a consequence of the approach taken in the proof, and it remains to be seen whether this bound is sharp. Note that,  coincidentally, it was shown in~\cite[Theorem 2]{casolo} that a finite soluble $\text{CA}_2$-group has derived length at most 4 and it was  demonstrated in~\cite[Example B]{casolo} that the bound is sharp. We comment also that our method of bounding the derived length by 4 is completely different from that of~ \cite[Example B]{casolo}. 

Our method also naturally extends to give the more general result for prosoluble groups with soluble centralisers:
\begin{theorem}\label{thm:prosoluble-with-large-derived-length}
    Let $G$ be a prosoluble $\textup{CA}_d$-group, for some positive integer $d$. Then $G$ is either  soluble or virtually pro-$p$.
\end{theorem}

We leave it as a question, as to whether the above result can be  correspondingly generalised to profinite $\textup{CA}_d$-groups;  that is, whether such a profinite group is either virtually pro-$p$ or virtually soluble. Note that a pro-$p$ $\textup{CA}_d$-group may not be soluble. For example,  a free pro-$p$ group~$F$ is a $\textup{CA}_d$-group since the centraliser of a  non-trivial element is procyclic.

\medskip

\noindent \textit{Notation.}  All subgroups of a profinite group are assumed to be closed and all automorphisms of a profinite group are assumed to be continuous. For a profinite group~$G$, we denote its order  by $|G|$ and we denote by $\pi(G)$ the union of prime divisors of the orders of finite quotients of $G$; these are  defined for example in \cite[Definition~2.1.1 and Section~2.3]{ProfiniteGroups} or \cite[Section~2.3]{rz}. In general, our notation is standard, following for example \cite{ProfiniteGroups} or \cite{rz}.

\medskip

\noindent \textit{Organisation.} In Section~\ref{sec:prelim} we collect together some preliminaries. In Section~\ref{sec:prosoluble} we prove Theorem~\ref{thm:prosoluble-with-large-derived-length} and set up the scene to prove Theorem~\ref{thm:main}, the proof of which we complete in Section~\ref{sec:main}.

\section{Preliminaries}\label{sec:prelim}

 First we list some standard facts about coprime actions, which are immediate from the corresponding facts for finite groups; see for instance \cite[Lemmas 4.28 and 4.29]{Isaacs} and \cite[Theorem 6.2.2(iv)]{Gorenstein}.
\begin{lemma}\label{lem:standard}
    Let $A$ be a profinite group acting by automorphisms on a profinite group~$G$ and suppose that $(|G|,|A|)=1$. Then
    \begin{enumerate}
        \item [\textup{(i)}] $G=[G,A]C_G(A)$;
          \item [\textup{(ii)}] $[G,A,A]=[G,A]$;
            \item [\textup{(iii)}] $C_{G/N}(A)=NC_G(A)/N$ for any $A$-invariant normal subgroup $N$ of $G$.
    \end{enumerate}
\end{lemma}

The following well-known lemma will be used often  throughout the sequel:
\begin{lemma}
   \label{lem:*}
   Let $G$ be a profinite group and $H$ a closed subgroup. Let $K$ be an open subgroup of $H$. Then $G$ has an open subgroup $L$ such that $L\cap H\leq K$. 
\end{lemma}

\begin{lemma}\label{known} Let $N$ be a normal subgroup of a profinite group $G$, and let $x\in G$ be an element such that $(|x|,|N|)=1$. Then we have $C_{G/N}(x)=C_G(x)N/N$.
\end{lemma}

\begin{proof} Let $T$ be the inverse image of $C_{G/N}(x)$ and $U :=C_G(x)N$. We need to show that $T=U$.  Obviously $U\leq T$ so it  suffices to establish that $T\leq U$.  To this end, observe that an element $y$ of $T$ normalises $\langle x\rangle$ if and only if $y$ centralises $x$. In other words,  we have $N_T(\langle x\rangle)=C_T(x)$. Let $\pi :=\pi(N)$. Note that $N\langle x\rangle$ is a normal subgroup of $T$ and  that $\langle x\rangle$ is a maximal $\pi'$-subgroup of $N\langle x\rangle$. By the Schur-Zassenhaus theorem the maximal $\pi'$-subgroups are conjugate in $N\langle x\rangle$. Thus, by the Frattini argument,  we obtain $T=N\langle x\rangle N_T(\langle x\rangle)$ and so $T=NC_T(x)$. It follows that $T\leq U$, as required.
\end{proof}

\begin{lemma}\label{rrrr} Let $G=BA$ be a profinite group, which is a product of a normal subgroup~$B$ and a subgroup $A$ such that $(|A|,|B|)=1$. Then $A$ has an open subgroup~$A_0$ such that $C_B(A_0)\neq1$. Hence, if $G$ is a $\textup{CA}_d$-group, then $A$ is virtually soluble of derived length at most $d$.  If furthermore $A$ is infinite, then $B$ is soluble of derived length at most~$d$.
\end{lemma}

\begin{proof} 
We proceed as in  \cite[Proof of Lemma 3.1(1)]{ccm}.  Let $N$ be any open proper subgroup of $B$ that is also normal in $G$. As the action of $A$ induces a finite group of automorphisms of $B/N$, there is an open subgroup $A_0$ of $A$ that acts trivially on $B/N$. By Lemma~\ref{lem:standard}(iii), we have $B=NC_B(A_0)$. As $A_0$ centralises a non-trivial element in $B$, we obtain that $A_0$ is soluble of derived length at most $d$. Also if $A$ is infinite then from $B=NC_B(A_0)$ we see that  $B/N$ is soluble of derived length at most $d$. As this holds for every open subgroup of $B$ that is normal in $G$, we have that $B$ is soluble of derived length at most $d$. 
\end{proof}

 For a profinite group~$G$, let $\pi$ be a subset of $\pi(G)$. Recall that the unique maximal normal pro-$\pi$ subgroup of~$G$ is denoted by $\mathcal{O}_\pi(G)$, and we denote by $\mathcal{O}_{\pi,\pi'}(G)$ the subgroup of $G$ such that $\mathcal{O}_{\pi,\pi'}(G)/\mathcal{O}_{\pi}(G)=\mathcal{O}_{\pi'}(G/\mathcal{O}_{\pi}(G))$; cf. \cite[Section 9.1]{Robinson}. Recall  also the following well-known facts, where we include proofs for the reader's convenience.

\begin{lemma}\label{lem:soluble-centraliser}
Let $G$ be a soluble profinite group  and let $\pi\subseteq \pi(G)$. Then  the following hold:
\begin{enumerate}
    \item [\textup{(i)}] if $\mathcal{O}_{\pi'}(G)=1$, then $C_G(\mathcal{O}_{\pi}(G))\leq \mathcal{O}_{\pi}(G)$;
    \item [\textup{(ii)}] if $G$ has an abelian $\pi$-Hall subgroup~$H$, then $H\leq \mathcal{O}_{\pi',\pi}(G)$.
\end{enumerate}
\end{lemma}

\begin{proof}
(i)  Let $M :=\mathcal{O}_{\pi}(G)$ and suppose $C:=C_G(M)\not\leq M$.  Let $N :=\mathcal{O}_{\pi,\pi'}(C)$. Then $N=\mathcal{O}_\pi(N)\times \mathcal{O}_{\pi'}(N)$,  and as $\mathcal{O}_{\pi'}(N)\ne 1$, we arrive at a contradiction to the assumption that $\mathcal{O}_{\pi'}(G)=1$.

(ii)    By taking the quotient if necessary, we may assume that $\mathcal{O}_{\pi'}(G)=1$. Let $M :=\mathcal{O}_{\pi}(G)$. We have $M\leq H$ and $C_G(M)\leq M$ by the previous part. Since $H$ is abelian, it follows that $M=H$. So  $H\leq \mathcal{O}_{\pi',\pi}(G)$.
\end{proof}

For a prosoluble group $G$, recall that the \emph{Fitting height} $h(G)$ is the length of a shortest series of normal subgroups of $G$ all of whose quotients are pronilpotent. Note that $h(G)$ is finite if and only if $G$ is an inverse limit of finite soluble groups of bounded Fitting height. We recall the following useful results for finite groups:

\begin{lemma}\cite[Theorem 1.1]{jaba}\label{lem:jaba}
Let $G=AB$ be a finite soluble group, where $A$ and $B$ are proper subgroups with $(|A|,|B|)=1$. Then $h(G)\le h(A)+h(B)+4d(B)-1$, where $d(B)$ denotes the derived length of $B$.
\end{lemma}

\begin{lemma}\cite[Theorem 1]{casolo}\label{lem:casolo}
Let $G$ be a finite soluble $\textup{CA}_d$-group, for a positive integer $d$. Then $h(G)\le 3d+2$.
\end{lemma}

\smallskip

Next, recall that a \emph{Sylow system} in a  group $G$ is a collection $\{P_1,P_2,\ldots\}$ of $p_i$-Sylow subgroups $P_i$ of $G$, one for each prime $p_i$ and such that $P_iP_j=P_jP_i$ for all $i$ and $j$. Then the \emph{system  normaliser} of such a Sylow system in $G$ is 
\[
\bigcap_i N_G(P_i).
\]
If $G$ is a profinite group and if $T$ is a system normaliser in $G$, then $T$ is pronilpotent (cf. \cite[Theorem~9.2.4]{Robinson}) and $G=T\gamma_\infty(G)$, where $\gamma_\infty(G)$ is the intersection of the terms of the lower central series of $G$. By \cite[Proposition 2.3.9]{rz}, every prosoluble group $G$ possesses a Sylow system and any two system normalisers in $G$ are conjugate (cf. also \cite[Theorem~9.2.4]{Robinson}).

\section{Prosoluble groups with soluble centralisers}\label{sec:prosoluble}

Here we prove Theorem~\ref{thm:prosoluble-with-large-derived-length} and establish other useful results for the next section.
\begin{theorem}\label{thm:prosoluble-with-metabelian-centr}
    Let $G$ be a prosoluble $\textup{CA}_d$-group for a positive integer $d$, such that $h(G)<\infty$.      Then $G$ is either soluble or virtually pro-$p$.
\end{theorem}

\begin{proof}
We proceed by induction on $h:=h(G)$. If $h=1$, this is clear since a pronilpotent group is a Cartesian product of its Sylow subgroups. 

Suppose next that $h=2$. Write $G=TB$, where $T$ is a system normaliser and $B=\gamma_\infty(G)$. Since $h=2$, both $T$ and $B$ are pronilpotent. Hence if $T$ is not a pro-$p$ group for some prime $p$, then $T$ is soluble. Similarly for $B$. If both $T$ and $B$ are virtually pro-$p$ for the same prime~$p$, then $G$ is virtually pro-$p$. Otherwise we use the fact that whenever $P$ is a $p$-Sylow subgroup of~$B$ and $H$ is a $p'$-Hall subgroup of~$T$, we have that $H$ acts on $P$ and so by Lemma~\ref{rrrr}, we obtain that $H$ is virtually soluble. If $B\ne P$, then as indicated above $B$ is soluble and we may also assume that $T=H$ else we similarly have that $T$ is soluble. So $G$ is virtually soluble. Note that since $G$ is prosoluble, if $G$ is virtually soluble, it follows that $G$ is soluble. Indeed, if $K$ is a  finite-index soluble subgroup, its normal core  $N$ is an open normal soluble subgroup, and then the quotient $G/N$ is of course soluble. So it remains to assume that $B=P$. If $H$ is infinite, Lemma~\ref{rrrr} yields that $B$ is soluble, as so we done as before. If $H$ is finite, then $G$ is virtually a pro-$p$ group.

Now assume $h\ge 3$, and choose a Sylow system $\{P_1,P_2,\ldots\}$ in $G$. Let $T_1$ be the system normaliser. Further let $B_1=\gamma_\infty(G)$ and recursively $B_{i+1}=\gamma_\infty(B_i)$ for $i=1,2,\ldots,h$. Then let $T_{i+1}$ be the system normaliser in $B_i$ corresponding to the Sylow  system $\{P_1\cap B_i, P_2\cap B_i,\ldots\}$. We have
\[
G=T_1B_1=T_1T_2B_2=T_1T_2T_3B_3=\ldots =T_1T_2\cdots T_h,
\]
where the subgroups $T_i$ are pronilpotent and for any $i\le j$, the subgroup $T_i$ normalises $T_j$. Note that $h(T_iT_{i+1})=2$ for any $i$. This is by construction, since $h(T_iT_{i+1}\cdots T_h)=h-i$ and $T_{i+2}\cdots T_h=\gamma_\infty(\gamma_\infty(T_i\cdots T_h))$. Therefore, by the case $h=2$, each product $T_iT_{i+1}$ is either soluble or virtually pro-$p$. If all are soluble, then so is $G$. If some $T_iT_{i+1}$ is virtually pro-$p$, by Lemma~\ref{lem:*} we have an open subgroup $K$ of $G$ for  which the intersection $K\cap T_iT_{i+1}$ is a pro-$p$ group. But then $h(K)\le h-1$, and so by induction the result follows. 
\end{proof}

In \cite[Lemma 3.5]{zapa}, it was shown that if $G$ is a prosoluble CA-group, then $h(G)\le 5$. We have the following:

\begin{lemma}\label{h(G)} Let $G$ be a prosoluble $\textup{CA}_d$-group for  $d\ge 2$. Then $h(G)\leq 8d+2$. 
\end{lemma}

\begin{proof} 
We may assume that $G$ is not a pro-$p$ group, else we are done. Following the notation in the  last paragraph of the previous proof, we let $T_1$ be the system normaliser of a Sylow system $\{P_1,P_2,\ldots\}$ in $G$.  

If $T_1$ has an infinite $p$-Sylow subgroup~$P$, then referring to the given Sylow system we deduce  by Lemma~\ref{rrrr} that $G$ has a soluble $p'$-Hall subgroup~$H$ of derived length at most $d$. Since $G=P_iH$ for some $i$ such that $P_i$ is the $p$-Sylow subgroup of~$G$ containing~$P$,  this yields by Lemma~\ref{lem:jaba}  that the Fitting height of every finite quotient of $G$ is at most $5d$. Therefore $h(G)\leq 5d$. 

Suppose now that $T_1$ is infinite but all Sylow subgroups of~$T_1$ are finite. Let $\pi\subseteq\pi(G)$ be any finite set of primes. Let $A$ be the infinite $\pi'$-Hall subgroup of $T_1$. Note that $A$ normalises a $\pi$-Hall subgroup of~$G$ and so, using Lemma~\ref{rrrr} as before, we conclude that for any finite set of primes $\pi\subseteq\pi(G)$ the group $G$ has a soluble $\pi$-Hall subgroup of derived length at most $d$. Now, writing $\pi(G)=\{p_1,p_2,\ldots\}$, we set: \[
\pi_1=\{p_1\},\quad \pi_2=\{p_1,p_2\},\quad\ldots\,\,,\quad \pi_i=\{p_1,p_2,\ldots,p_i\}, \quad\ldots
\]
So $\pi(G)=\cup \pi_i$. We then choose $\pi_i$-Hall subgroups $H_i$ such that 
\[
H_1\leq H_2\leq \cdots
\]
and define the abstract subgroup $H=\cup H_i$. This abstract subgroup $H$ is dense in $G$. Since $H$ is a union of soluble groups of derived length at most $d$, it follows that $G$ is soluble of derived length at most~$d$. 

Finally, if $T_1$ is finite, still using notation in the  last paragraph of the previous proof, we consider the system normaliser~$T_2$ of $B_1=\gamma_\infty(G)$, corresponding to the Sylow system $\{P_1\cap B_1, P_2\cap B_1,\ldots\}$, and proceed as in \cite[Lemma~3.5]{zapa}. Specifically, if $T_2$ is infinite, we are done as before since we obtain $h(B_1)\le 5d$, so suppose $T_2$ is finite. Now, by  Lemma~\ref{lem:casolo} a finite soluble $\textup{CA}_d$-group has Fitting height at most $3d+2$. Hence  it follows that for some $k\le 3d+2$, we have that $T_{k+1}$ is infinite, and the earlier part shows that $h(B_k)\le 5d$.  Thus the result follows.
\end{proof}

Hence Theorem~\ref{thm:prosoluble-with-large-derived-length} is now proved. We remark that, for our purposes,  the best bound on the Fitting height in the above lemma is not necessary. One can see however, that the bound can be clearly improved for the case $d=2$ using \cite[Theorem~2]{casolo}.

\smallskip

For a profinite group $G$, we recall next that the \emph{Fitting subgroup}~$F(G)$ of~$G$ is the maximal pronilpotent normal subgroup of~$G$.

\begin{lemma}\label{bbaa-generalised} If a pro-$\pi$ group $P$ acts on a pro-$\pi'$ group $Q$, and if  $QP$ is a $\textup{CA}_2$-group, then $P$ has an open subgroup $H$ such that $H'$ centralises $Q$. 
\end{lemma}

\begin{proof}
If $P$ is finite, then we can take $H$ to be the trivial subgroup. So we may assume that $P$ is infinite. Then Lemma~\ref{rrrr} gives that $Q$ is metabelian and $P$ is virtually metabelian. 

Suppose first that  $Q'\neq1$, and we look at how $P$ acts on $Q'$. Again by Lemma~\ref{rrrr} there is an open subgroup $T\leq P$ such that $C_{Q'}(T)\neq1$. Thus $Q'T$ is contained in a centraliser. So $T'\leq F(Q'T)$ and hence $[Q',T']=1$. Since $Q'$ is normal in $QP$, it follows that $[Q,T']$ centralises~$Q'$. So $[Q,T']$ and $T'$ are both contained in the centraliser of $Q'$. Then $[Q,T',T']$  is abelian as it is contained in the derived subgroup of the centraliser of~$Q'$. By Lemma~\ref{lem:standard}(ii), it follows that $[Q,T']$ is abelian.  Now we consider the action of $T$ on the abelian group $[Q,T']$ and find an open subgroup $H\leq T$ such that $H'$ centralises $[Q,T']$. Then $1=[Q,H',H']=[Q,H']$,   and so $H'$ centralises $Q$, as required. 

Lastly, if $Q'=1$, then we proceed similarly by considering the action of $P$ on $Q$. We obtain as before an open subgroup $H\leq P$ such that $C_{Q}(H)\neq1$, and we deduce that  $[Q,H']=1$. Hence the result.
\end{proof}

\begin{theorem}\label{thm:length-4}
    Let $G$ be a soluble profinite $\textup{CA}_2$-group.     If $G$ is not virtually pro-$p$, then $G$ has an open subgroup of derived length  at most 4.
\end{theorem}

\begin{proof}
Suppose first that there is a prime $p$ such that $\mathcal{O}_p(G)\ne 1$ and that $\mathcal{O}_{p'}(G)=1$. Note that if $\mathcal{O}_p(G)$ is finite then $C_G(\mathcal{O}_p(G))$ is open and so $G$ is virtually metabelian. So we may assume that $\mathcal{O}_p(G)$ is infinite. Let $H$ be a $p'$-Hall subgroup of~$G$. Note that $H$ cannot be finite, as then $G$ would be virtually pro-$p$.  Then we may assume in light of Lemmas~\ref{rrrr} and \ref{lem:*} that $\mathcal{O}_p(G)$ and $H$ are metabelian. By replacing $G$ with an open  subgroup in virtue of Lemma~\ref{lem:*}, by Lemma~\ref{bbaa-generalised} we can assume that $H'$ centralises~$\mathcal{O}_p(G)$. Hence by Lemma~\ref{lem:soluble-centraliser}(i) it follows that $H'=1$, and so $H$ is abelian. We further have that $H\le \mathcal{O}_{p,p'}(G)$ by Lemma~\ref{lem:soluble-centraliser}(ii). Moving to the quotient $\overline{G}:=G/\mathcal{O}_p(G)$, we now have $\overline{H}\leq \mathcal{O}_{p'}(\overline{G})$, and since $\overline{H}$ is a $p'$-Hall  subgroup, we then in fact have   $\overline{H}= \mathcal{O}_{p'}(\overline{G})$.  It follows that $\mathcal{O}_p(G)H$ is normal in $G$. Next, let $P$ be a $p$-Sylow subgroup of $N_G(H)$. By the Frattini argument $G=\mathcal{O}_p(G)HN_G(H)=\mathcal{O}_p(G)HP$, where the last equality follows since the index of $\mathcal{O}_p(G)H$ in $G$ is a $p$-power. Note that $P'$ centralises $H$; cf. Lemmas~\ref{lem:*} and \ref{bbaa-generalised}. However, in $\overline{G}$, we have $C_{\overline{G}}(\overline{H})\leq \overline{H}$.  Therefore $P'\leq \mathcal{O}_p(G)$, and  from considering $G\trianglerighteq \mathcal{O}_p(G)H\trianglerighteq \mathcal{O}_p(G)\trianglerighteq  \mathcal{O}_p(G)'$ we see that the result follows in this case.

     Suppose now that there are two disjoint sets of primes $\pi_1$ and $\pi_2$ such that $\mathcal{O}_{\pi_1}(G)$ and $\mathcal{O}_{\pi_2}(G)$ are both non-trivial. Without loss of generality we may assume that $\pi(G)=\pi_1\cup\pi_2$. Set $P_1:=\mathcal{O}_{\pi_1}(G)$ and $Q_1:=\mathcal{O}_{\pi_2}(G)$. For $N:=P_1\times Q_1$, we let $P_2N/N:=\mathcal{O}_{\pi_1}(G/N)$ and set $K:=P_2N$. Here $P_2$ is a $\pi_1$-Hall subgroup of $K$. By the Frattini argument $G=KN_G(P_2)$. Next, for $H$ a $\pi_2$-Hall subgroup of $N_G(P_2)$, we note that $Q_1H$ is a $\pi_2$-Hall subgroup of~$G$. Using Lemmas~\ref{lem:*} and \ref{bbaa-generalised} we can assume that $H'$ centralises~$P_2$. Then, writing $\overline{G}:=G/Q_1$, we observe that $\mathcal{O}_{\pi_2}(\overline{G})=1$ and $\overline{P}_2=\mathcal{O}_{\pi_1}(\overline{G})$. Because of Lemma~\ref{lem:soluble-centraliser}(i) we conclude that $H'\leq Q_1$. So the $\pi_2$-Hall subgroup $Q_1H$ is abelian modulo $N$. Similarly, a $\pi_1$-Hall subgroup of $G$ is abelian modulo $N$. By It\^{o}'s theorem~\cite{Ito} (see also \cite[Exercise 1.6.21]{ProfiniteGroups}), the quotient group $G/N$ is metabelian. Also $N$ is metabelian. Hence, in light of our repeated use of Lemma~\ref{lem:*}, we see that $G$ has an open subgroup of derived length at most 4.
\end{proof}

\section{Proof of Theorem~\ref{thm:main}}\label{sec:main}
We begin with the following useful result.

\begin{lemma}\label{unknown} Let $G$ be a profinite $\textup{CA}_2$-group. Suppose $G$ has a pronilpotent normal subgroup $N$ such that $G/N$ is isomorphic to a Cartesian product of finite non-abelian simple groups. Then $G$ is virtually pronilpotent.
\end{lemma}

\begin{proof} Assume that the lemma is false; i.e. that $G/N$ is infinite. Thus $G$ has infinitely many normal subgroups $S_i$, each containing $N$, such that $S_i/N$ is non-abelian simple and $G=\prod_{i\in I}S_i$ for some index set~$I$.  We claim that each subgroup $S_i$ contains a subgroup $D_i$ whose image in $S_i/N$ is isomorphic to the dihedral group of order 2$p_i$, where $p_i$ is an odd prime. Indeed, using the bar notation for elements in the quotient $S_i/N$, by the Feit-Thompson theorem and since $S_i/N$ is simple, there is an element $a\in S_i$ such that $\overline{a}$ is of order 2. Then by the Baer-Suzuki theorem, there is an element $\overline{b}\in S_i/N$ such that $\langle \overline{a}, \overline{a}^{\overline{b}}\rangle$ is not a 2-group. So $\langle \overline{a}, \overline{a}^{\overline{b}}\rangle$ contains a dihedral subgroup as required. 

For a subset $J\subseteq I$, choose a minimal subgroup $D_J$ subject to the condition that its image modulo $N$ is the same as that of $\langle D_i\mid i\in J\rangle$. Note that $D_J\cap N\leq\Phi(D_J)$; cf. \cite[Lemma 2.8.15]{rz}. 

It follows that $\pi(D_J\cap N)\subseteq\pi(D_JN/N)$. Indeed, suppose that this is false and let $p\in\pi(D_J\cap N)$ while $p\not\in\pi(D_JN/N)$. Let $P$ be the $p$-Sylow subgroup of $D_J\cap N$ and $H$ a $p'$-Hall subgroup of $D_J$. We see that $D_J=PH$ and $HN=D_JN$. Because of minimality we are forced to conclude that $H=D_J$, a contradiction. Thus for every subset $J$ of the set $I$, we have $\pi(D_J\cap N)\subseteq\pi(D_JN/N)$.

Also, observe that $\pi(G/N)\subseteq\pi(N)$. Indeed, suppose that this is false and let $p\in\pi(G/N)$ while $p\not\in\pi(N)$. There is an index $i$ such that $S_i$ contains a non-trivial $p$-element $a$. Since $(|a|,|N|)=1$, in view of Lemma \ref{known} we have $C_{G/N}(a)=C_G(a)N/N$. Note that $C_{G/N}(a)$ contains a Cartesian product of infinitely many finite non-abelian simple groups while $C_G(a)$ is metabelian. This contradiction shows that $\pi(G/N)\subseteq\pi(N)$. 

Note also that $N$ is metabelian, since the centralisers in $G$ are metabelian and since different Sylow subgroups of $N$ commute. We will now show that there is a prime~$q$ and an infinite set of indices $J$ such that $q\not\in\pi(D_J)$ and $G$ possesses a non-trivial normal abelian pro-$q$ subgroup~$Q$.

At first, suppose that there is a prime $p$ such that $p_i=p$ for infinitely many indices $i$. Let $J$ be the set of all indices $i$ such that $p_i=p$. Then $\pi(D_J)=\{2,p\}$. Since $\pi(G/N)\subseteq\pi(N)$ and since the order of any finite non-abelian simple group is divisible by at least three primes, there exists an odd prime $q\in\pi(N)$ such that $q\not=p$.  Note that $N$ is the Fitting subgroup of $G$ which is characteristic. Hence since $N$ is metabelian, the group $G$ has a non-trivial normal abelian pro-$q$ subgroup $Q$.

Now suppose that for any prime $p$ there are only finitely many indices $i$ such that $p_i=p$. Pick any odd prime $q\in\pi(N)$. As before, the group~$G$ has a non-trivial normal abelian pro-$q$ subgroup $Q$. Let $J$ be the set of all indices $i$ with the property that $p_i\neq q$. Thus, again  $q\not\in\pi(D_J)$, as required.

By Lemma \ref{rrrr}, the subgroup $D_J$ has an open normal subgroup $K$ such that $C_Q(K)\neq1$. As $Q$ is abelian, we deduce that $QK\leq C_G(C_Q(K))$ and hence $QK$ is metabelian. Taking into account $q\not\in\pi(K)$, it follows that the commutator subgroup $K'$ centralises $Q$. Then so does $[G,K']$ and therefore $[G,K']$ is metabelian. Since $G/N$ is a product of finite non-abelian simple groups, in particular it has no normal metabelian subgroups and so $[G,K']\le N$. Equivalently, the image of $K'$ in $G/N$ is central. It follows that $K'\leq N$. But then the image of $D_J$ in $G/N$ is virtually abelian. This yields the desired  contradiction since the image of $D_J$ in $G/N$ is a Cartesian product of infinitely many non-abelian dihedral groups.
\end{proof}

We say that $G$ is \emph{semisimple} if it is isomorphic to a Cartesian product of  finite non-abelian simple groups. Recall that for a finite group $G$, the \emph{non-soluble length} $\lambda(G)$ of $G$ is the minimal number of non-soluble factors in a normal series of $G$,  each of whose factors is either soluble or semisimple. 
We are now in position to prove our main result:
\begin{proof}[Proof of Theorem \ref{thm:main}]
From Lemma~\ref{h(G)}, we have that any prosoluble subgroup~$H$ of~$G$ satisfies $h(H)\le 18$. Hence,  in view of 
\cite[Corollary 1.2]{khushu}, we have that $\lambda(\overline{G})\le 18$ for any finite quotient $\overline{G}$ of $G$. Indeed, if $\overline{G}=G/N$ and $KN/N$ is soluble, by \cite[Lemma 2.8.15]{rz} there is a prosoluble subgroup~$H$ such that $HN/N=KN/N$. 

Then by \cite[Lemmas~2 and 3]{Wilson}, there is a  finite normal series for $G$, all of whose factors are either prosoluble or semisimple. We proceed by induction on the length $\ell:=\ell(G)$ of this normal series. If $\ell=1$, then $G$ is either prosoluble or semisimple. In the former case,  we are done by Theorems~\ref{thm:prosoluble-with-metabelian-centr}   and \ref{thm:length-4}, and in the latter it follows that $G$ is finite.

Suppose now that the result is true for $\ell$, and consider a profinite group $G$ with a  normal series of length $\ell+1$, all of whose factors are either prosoluble or semisimple:
\[
1=G_0\le G_1\le \cdots \le G_\ell\le G_{\ell+1}=G.
\]
Suppose first that $G/G_\ell$ is prosoluble. By induction $G_\ell$ is virtually pro-$p$ or virtually soluble. Hence, using Lemma~\ref{lem:*}, it follows that $G$ is virtually prosoluble. We are then done upon applying Theorem~\ref{thm:length-4}. Finally we suppose that $G/G_\ell$ is semisimple. By applying Lemma~\ref{lem:*}, we may assume that $G$ is prosoluble-by-semisimple.  From \cite[Lemma 2.8.15]{rz}, there is a minimal subgroup $K$ such that $G=G_\ell K$ and also $K\cap G_\ell \leq \Phi(K)$. Recalling that $\Phi(K)$ is pronilpotent, we therefore obtain that $K$ is pronilpotent-by-semisimple. Then Lemma~\ref{unknown} yields that $K$ is virtually pronilpotent. As before Lemma~\ref{lem:*} and Theorem~\ref{thm:length-4} completes the proof.
\end{proof}

\end{document}